\documentclass[a4paper]{article}

\usepackage[english]{babel}
\usepackage[utf8x]{inputenc}
\usepackage{geometry} 
\usepackage{euscript}

\usepackage[pagebackref,colorlinks,linkcolor=blue,citecolor=blue,urlcolor=blue,hypertexnames=true]{hyperref}
\usepackage{amsrefs}

\usepackage{graphicx} 
\usepackage{epstopdf}
\usepackage{amscd}
\usepackage{array} 
\usepackage{verbatim} 
\usepackage{amssymb}
\usepackage{amsmath}
\usepackage{amsthm}
\usepackage{csquotes}
\usepackage{mathtools}

\newtheorem{definition}{Definition}
\newtheorem{remark}{Remark}
\newtheorem{theorem}{Theorem}
\newtheorem{lemma}{Lemma}
\newtheorem{proposition}{Proposition}

\newtheorem{cor}{Corollary}

\def\pione{\pi_1(\mathcal{G},T)}
\def\overpione{\overline{\pi}_1(\mathcal{G},T)}
\def\deriv{\mathbf{d}_{\mathcal{G}}}

\title{Virtually free groups are $p$-Schatten stable}

\author{Maria Gerasimova \and
  Konstantin Shchepin}
\newcommand{\Addresses}{{
  \bigskip
  \footnotesize
  \medskip
  Maria Gerasimova, \textsc{WWU M{\"u}nster}\par\nopagebreak
  \textit{E-mail address} : \texttt{mari9gerasimova@mail.ru}
  
\medskip

  Konstantin Shchepin \textsc{}\par\nopagebreak
  \textit{E-mail address} : \texttt{shchepin.konstantin@gmail.com}
}}

\begin{document}
\vspace{-8ex}
\date{}
\maketitle

\begin{abstract}

In \cite{lazarovich2021virtually} was proven that finitely generated virtually free groups are stable in permutations. In this note we show that a similar strategy can be used to prove that finitely generated virtually free groups are stable with respect to a normalized $p$-Schatten norm for $1\leq p < \infty$. In particular, this implies that virtually free groups are Hilbert-Schmidt stable.

\end{abstract}

\section{Introduction}
Given a system of equations in noncommutative  variables one can ask if it is \textquote{stable}, meaning that each of its \textquote{almost} solutions is \textquote{close} to an actual solution. 
For variables $s_i$ and equations of the form $\prod_k s_{i_k}=1$ this \textquote{stability} property is a property of the group generated by $s_i$ with relations given by these equations -- it does not depend on the sets of generators and relations.  
To make this into a concrete problem one needs to specify what exactly is considered as a solution and what \textquote{almost} and \textquote{close} mean. 
For example, as a solution one can consider a set of permutations or a set of matrices satisfying some additional conditions. 
In this paper we will focus on the set of unitary matrices of arbitrary size equipped with a normalized $p$-Schatten norm for $1\leq p<\infty$. 
When $p=2$ this norm is called the Hilbert-Schmidt norm.

The systematic study of stability with respect to the Hilbert-Schmidt norm was taken in \cite{hadwin2018stability}, where  in particular was given a characterization of stability of amenable groups in terms of the approximation property for characters. 
More precisely, it was proven that an amenable group is stable with respect to the Hilbert-Schmidt norm if and only if each character is a pointwise limit of traces of finite-dimensional representations. 
Using this characterization it was shown that all finitely generated virtually abelian groups as well as the discrete Heisenberg group $\mathbb{H}^3$ are Hilbert-Schmidt stable. 
For non-amenable groups this approximation property was shown to be only a necessary condition. In particular, $SL_3(\mathbb{Z})$ satisfies this property, but is not Hilbert-Schmidt stable (see \cite{hadwin2018stability},\cite{becker2020group}).

Let us observe that Hilbert-Schmidt stability is preserved under taking free products (obviously) and direct products with Hilbert-Schmidt stable amenable groups (see \cite{ioana2019ii1}). Other examples of Hilbert-Schmidt stable groups include all one-relator groups with non-trivial center (see \cite{hadwin2018stability}).

In \cite{lazarovich2021virtually} was proven that finitely generated virtually free groups are stable in permutations. 
Although the normalized Hamming distance can be expressed using the Hilbert-Schmidt distance,  this cannot be used directly to deduce Hilbert-Schmidt stability from stability in permutations. 
Despite this, we show that the similar strategy can be used to prove that finitely generated virtually free groups are stable with respect to a $p$-Schatten norm for $1\leq p<\infty$. In particular, they are Hilbert-Schmidt stable.

Let us mention that in \cite{eilers2020c} was shown that for an amenable group operator norm stability implies Hilbert-Schmidt stability.  Virtually free groups are known to be operator norm stable, but since they are not amenable it does not give us Hilbert-Schmidt stability.

%Let us note that virtually free groups are known to be operator norm stable, but since they are not amenable it does not give us Hilbert-Schmidt stability.

%SECTION - PRELIMINARIES
\section{Preliminaries}
\subsection{$p$-Schatten stability of finite groups}
Let $W$ be some finite-dimensional Hilbert space and let us fix some $1\leq p <\infty$.

\begin{definition}
For any operator $A\colon W \to W$ let us define the $p$-Schatten norm by 
$$||A||_{p}\coloneqq(tr(|A|^p))^{\frac{1}{p}},$$
where $|A|=\sqrt{A^*A}$.
\end{definition}

\begin{definition}
For any operator $A\colon W \to W$ let us define the normalized  $p$-Schatten norm by 
$$||A||'_{p}\coloneqq \left(tr \left(\frac{|A|^p}{\dim W}\right)\right)^{\frac{1}{p}}.$$
\end{definition}

Let us note that $||\cdot||_2$ is called the \textit{Frobenius norm} and $||\cdot||'_2$ is called the \textit{Hilbert-Schmidt norm}. \\
Recall that any operator $A\colon W\to W$ can be decomposed as $$A=U^*\Lambda V,$$ where $U,V\in {\rm U}(W)$ and 
$\Lambda=diag(\lambda_i)$, $\lambda_i \in \mathbb{R}_{\ge0}$. $\lambda_i$ are called \textit{singular values} of $A$ and they are uniquely defined up to a permutation.\\
One can compute the $p$-Schatten norm as 
$$||A||_p=\left(\sum_i \lambda_i^p\right)^{\frac{1}{p}}.$$ 
From now on let $G=\langle S, R \rangle$ be a finitely presented group, $F_S$ be the free group generated by $S$. We recall some basic definitions.

\begin{definition}
A sequence of homomorphisms $\varphi_n\colon F_S\to U(W_n)$, where $W_n$ is a sequence of finite-dimensional Hilbert spaces, is called an asymptotic homomorphism of $G$ with respect to the normalized $p$-Schatten norm if for all $r\in R$
$$\lim_n||\varphi_n(r)-I_n||'_p=0,$$
where $I_n$ is an identity operator on $W_n$. 
\end{definition}

\begin{definition}
We will say that a group $G$ is stable with respect to the normalized $p$-Schatten norm if for any asymptotic homomorphism $\rho_n\colon F_S\to U(W_n)$ there exists a sequence of representations $\rho'_n\colon G\to U(W_n)$ such that for all $s\in S$
$$ \lim_n ||\rho_n(s)-\rho'_n(s)||'_p=0.$$ 
\end{definition}

\begin{definition}
We will say that a group $G$ is flexibly stable with respect to the normalized \mbox{$p$-Schatten} norm if for any asymptotic homomorphism $\rho_n\colon F_S\to U(W_n)$ there exists a sequence of representations $\rho'_n\colon G\to U(W_n')$, $W_n\subset W'_n$ such that for all $s\in S$
$$ \lim_n ||\rho_n(s)-P\rho'_n(s)P||'_p=0,$$ 
where $P$ is an orthogonal projection $P\colon W_n'\to W_n$ and $\frac{\dim(W'_n)}{\dim(W_n)} \to 1$ as $n \to \infty$.
\end{definition}

Flexible stability of finite groups with respect to a normalized $p$-Schatten norm for $1\leq p \leq 2$ follows from Theorem 6.9 in \cite{gowers2017inverse}. Later the stability result of \cite{gowers2017inverse} was generalized in \cite{de2019operator} to the class of amenable groups with respect to unitary groups of von Neumann algebras equipped with any unitarily invariant, ultraweakly lower semi-continuous semi-norm. In particular, this gives flexible stability of finite groups with respect to a normalized $p$-Schatten norm for $1\leq p <\infty$. \\
One can show that in the case of finite groups flexible stability with respect to a $p$-Schatten norm implies stability. For $p=2$ it follows from Theorem 3.2 in \cite{akhtiamov2020uniform}. This theorem can be generalized for all ${1\leq p < \infty}$ using the basic inequality $ ||BAC^*||_p \leq ||A||_p||B||_{op}||C||_{op}$, where ${A\in M_{m\times m}(\mathbb{C})}$ and ${B,C \in M_{n\times m}(\mathbb{C})}$ (see Lemma 6.1 in \cite{gowers2017inverse}). 
\begin{cor}
Any finite group is stable with respect to a normalized $p$-Schatten norm for ${1\leq p <\infty}$.
\end{cor}

\subsection{ Graph of groups}

We will briefly recall the main definitions and notations which we will use further.
Using Serre's notation for graphs, we will say that a graph $\Gamma$ consists of a set vertices $V(\Gamma)$ and a set of oriented edges $E(\Gamma)$, moreover, each edge has an origin $o(e)\in V(\Gamma)$ and a terminus $t(e)\in V(\Gamma)$ and admits a distinct opposite edge $\overline{e}\in E(\Gamma)$ that satisfies $\overline{\overline{e}}=e$, $t(\overline{e})=o(e)$ and $o(\overline{e})=t(e)$. In this notation an orientation of the graph $\Gamma$ is just a subset $\vec{E}(\Gamma)\subset E(\Gamma)$ containing exactly one edge from each pair of opposite edges $\{e,\bar{e}\}$. 

\begin{definition}
A graph of groups $\mathcal{G}$ is 
$$\mathcal{G}=\left(\Gamma,\{G_v\}_{v\in V(\Gamma)},\{G_e\}_{e\in E(\Gamma)},\{i_e\colon G_e\to G_{t(e)} \}_{e \in E(\Gamma)}\right)$$
where $\Gamma$ is a connected graph, $G_v$ and $G_e$ -- vertex and edge groups correspondingly with $G_e=G_{\overline{e}}$ and $i_e \colon G_e\to G_{t(e)}$ are injective homomorphisms.  
\end{definition}

Let $\mathcal{G}$ be a graph of groups. Fix a subtree $T\subset \Gamma$ and an orientation $\vec{E}(\Gamma)$, this gives the induced orientation $\vec{E}(T)$. 
Consider the group $\overpione$ defined as the free product 
$$\overpione=*_{v\in V(\Gamma)}G_v*F\left(\{s_e\}_{e\in \vec{E}(\Gamma)}\right)$$ 
We will also fix the generating set of $\overpione$
$$S_{\mathcal{G}}=\bigcup_{v\in V(\Gamma)}G_v\cup \{s_e\}_{e\in\vec{E}(\Gamma)}.$$

\begin{definition}
The fundamental group $\pione$ of the graph of groups $\mathcal{G}$ with respect to the subtree $T$ is the quotient of the free product $\overpione$ by the normal subgroup generated by the relations
$$R_{\mathcal{G}}=
\begin{cases}
s_e  &  \forall e\in \vec{E}(T),\\
s_e^{-1}i_e(g_e)s_e\left(i_{\overline{e}}(g_e)\right)^{-1} & \forall e\in \vec{E}(\Gamma), \, g_e\in G_e.
\end{cases}
$$
\end{definition} 
In the definition above one can use any subtree $T$ of $\Gamma$ but for us $T$ will always be a maximal spanning tree because of the following results.

\begin{remark}
\label{indep}
The fundamental group $\pione$ as well as the group $\overpione$ are independent of the choice of a maximal spanning tree T up to isomorphism. 
\end{remark}

\begin{theorem}[Stallings, \cite{stallings1968torsion}]
The fundamental group $\pi_1(\mathcal{G},T)$ of a finite graph of groups $\mathcal{G}$ with finite vertex groups with respect to any maximal spanning tree $T$ is virtually free. Any finitely generated virtually free group can be constructed this way. 
\end{theorem}

\subsection{Stable epimorphisms}

One of the main definitions of \cite{lazarovich2021virtually} was one of a $P$-stable epimorphism.  We will need a natural analogue of this definition. 
Let $\overline{G}$ be a group generated by a finite set $S$, $N\triangleleft \overline{G}$ be a normal subgroup normally generated by some finite set $R\subset \overline{G}$. Denote $G=\overline{G}/N$. We say that a representation $\rho\colon \overline{G}\to U(W)$ is a $\delta$-almost $G$-representation if for all $r\in R$
$$||\rho(r)-I ||'_p< \delta.$$
For two maps $\rho_1,\rho_2\colon \overline{G}\to U(W)$ and a subset $A\subset \overline{G}$ we will denote by $$d_A(\rho_1,\rho_2)\colon=\max_{g\in A}||\rho_1(g)-\rho_2(g)||'_p.$$

\begin{definition}
We will say that the epimorphism $\phi\colon\overline{G}\to G$ is stable with respect to the normalized $p$-Schatten norm if for every $\varepsilon>0$ there is a $\delta>0$ such that for every $\delta$-almost $G$-representation $\rho\colon \overline{G}\to U(W)$ there is a $G$-representation $\rho'\colon G\to U(W)$ with $d_S(\rho,\phi^*(\rho'))< \varepsilon$.  
\end{definition}

For $\pione$ we will always use the generating set $S_\mathcal{G}$ and the set of relations $R_{\mathcal{G}}$. Similarly to \cite{lazarovich2021virtually} one can get the following lemma.

\begin{lemma}
The stability of the epimorphism is a well-defined notion (i.e. it is independent of the choices of the finite sets $S$ and $R$)
\end{lemma}

This definition is motivated by the following fact. 

\begin{remark}
If $\overline{G}$ and an epimorphism $\phi\colon\overline{G}\to G$ are stable with respect to some norm then $G$ is stable with respect to the same norm. 
\end{remark}

In our case all vertex groups of a graph of groups are finite and hence are stable with respect to the normalized $p$-Schatten norm. $\overpione$ is stable as a free product of stable groups. So to prove stability of virtually free groups we need only the following theorem. 

\begin{theorem}
\label{ourmain}
The epimorphism $\overpione \to \pione$ is stable with respect to a $p$-Schatten norm for all $1\le p<\infty$. 
\end{theorem}

\section{Stability of virtually free groups}

As we mentioned already, the strategy of proving that virtually free groups are stable with respect to the normalized $p$-Schatten norm is very similar to the strategy used in \cite{lazarovich2021virtually} to prove that virtually free groups are stable in permutations. Although the statements will look similar, the proofs will be vastly different.

\subsection{On  \textquote{close} representations of finite groups}

In \cite{gowers2017inverse} was proven that two representations of a finite group $G$ which are close to each other with respect to the normalized $p$-Schatten norm are almost isomorphic and there is an intertwining operator close to identity.
We will use a little bit more data hidden in the proof of this theorem and will change its statement accordingly (see Theorem 7.3 in \cite{gowers2017inverse}).

\begin{theorem} [Gowers, Hatami]
\label{main}
Let $\rho_1, \rho_2\colon G \to U(W)$ be two representations of a finite group $G$ such that $d_G(\rho_1,\rho_2)\le \delta$. 
Then for $i=1,2$ there exist $\rho_i$-invariant subspaces $V_i$ of dimension at least $\left(1-(2\delta)^p\right)\cdot \dim W$ 
and an operator $T$ such that $||T-I||'_p\le 3\delta$ and $\rho_2(x)\cdot T=T\cdot \rho_1(x)$ for every $x\in G$. 
Moreover, $T$ maps an orthonormal basis of $V_1$ to an orthonormal basis of $V_2$ and $V_1^{\perp}$ to $0$.
\end{theorem}

We will prove the following corollary which we will use further.

\begin{cor}
\label{maincor}
If moreover $\rho_1$ and $\rho_2$ are isomorphic, then there exists non-degenerate intertwining operator $T'\in U(W)$ such that $||T'-I||'_p\le 5\delta$.
\end{cor}

\begin{proof}
$T$ establishes an isomorphism between representations $\rho_1|_{V_1}$ and $\rho_2|_{V_2}$.   
Since $\rho_1$ and $\rho_2$ are also isomorphic, $\rho_1|_{V_1^{\perp}}\colon G\to U(V_1^{\perp})$ is isomorphic to $\rho_2|_{V_2^{\perp}}\colon G \to U(V_2^{\perp})$.  
Let ${S\colon V_1^{\perp} \to V_2^{\perp}}$ be any intertwining operator between these representations which maps an orthonormal basis to an orthonormal basis. 
We can extend $S$ to $S'\colon W \to W$ by ${S'(V_1)=0}$. 
An operator $T':=T+S'$ maps an orthonormal basis of  $W$ to an orthonormal basis of $W$, hence $T'$ is a unitary intertwining operator. 
Singular values of $S'$ are either $0$ or $1$ and at most $\dim (V_1^{\perp})$ of them are equal to $1$. But $\dim (V_1^{\perp})\le (2\delta)^p \cdot \dim W$ and hence we get:  
$$||T'-I||'_p\le||T-I||'_p+||S'||'_p\le 3\delta + \left((2\delta)^p\right)^{\frac{1}{p}}=5\delta.$$   
\end{proof}

We will also need the following lemma that might be considered as a converse statement to Theorem~\ref{main}. 

\begin{lemma}
\label{converse}
Let $G$ be a finite group, $\rho\colon G \to U(W_1)$ and $\sigma_1, \sigma_2\colon G \to U(W_2)$ be some representations. 
If for some $0<\delta<1$ we have $\dim(W_2)\le \delta^p\cdot \dim(W_1\oplus W_2)$, then $d_G(\rho\oplus\sigma_1, \rho\oplus\sigma_2)\le 2\delta$. 
\end{lemma}

\begin{proof}
For any $g\in G$, $(\rho\oplus\sigma_1)(g)-(\rho\oplus\sigma_2)(g)$ has at most $\dim(W_2)$ nonzero singular values and all of them are not greater than $2$. Hence,  
$$||(\rho\oplus\sigma_1)(g)-(\rho\oplus\sigma_2)(g)||'_p\le  \left( \frac{2^p\cdot \dim(W_2)}{\dim(W_1)+\dim(W_2)}  \right)^{\frac{1}{p}}\le 2\delta.$$
\end{proof}

\subsection{Cones and representations}

Since the representation theory of finite groups is nice, we can introduce cones of isomorphism classes of finite-dimensional representations, which are similar to the cones of isomorphism classes of actions on finite sets introduced in \cite{lazarovich2021virtually}.
\begin{definition}
Let $Repr(G)$ be the set of isomorphism classes of finite-dimensional unitary representations of a group $G$ and $Irr(G)$ be the set of isomorphism classes of irreducible finite-dimensional unitary representations of $G$. 
\end{definition}

\begin{definition}
Let $\Theta_G:=\bigoplus_{\pi \in Irr(G)} \mathbb{Z}\pi$ be the free $\mathbb{Z}$-module with basis $Irr(G)$ and let $\Theta^+_G \subset \Theta_G$ be its non-negative cone.
\end{definition}

Since every unitary representation is just a sum of irreducible representations, $Repr(G)$ can be identified with $\Theta_G^+$.
For a representation $\rho: G \to U(W)$ let us denote by $\rho^\#$ the corresponding element of $\Theta_G^+$.

We will consider a norm on $\Theta_G$, defined by
$$||\lambda||_G=\sum_{\pi \in Irr(G)}|\lambda_\pi|\cdot \dim \pi.$$
For a representation $\rho: G \to U(W)$, $||\rho^\#||_G=\dim(W)$.
For any $\lambda\in\Theta_G^+$ and any Hilbert space $W$ with $\dim(W)=||\lambda||_G$ there is a representation $\rho\colon G \to U(W)$ with $\rho^\#=\lambda$.

Given a homomorphism $i\colon H\to G$ and a representation $\rho\colon G \to U(W)$ one can consider the pullback representation $i^*(\rho)\colon H\to U(W)$. This gives a linear map $i^\# \colon \Theta_G \to \Theta_H$ which maps $\rho^\#$ to $(i^*(\rho))^\#$.

%NORM ON THIS SPACE, SAME FOR GRAPH OF GROUPS
For a graph of groups 
$$\mathcal{G}=(\Gamma,\{G_v\}_{v\in V(\Gamma)},\{G_e\}_{e\in E(\Gamma)},\{i_e:G_e\to G_{t(e)} \}_{e \in E(\Gamma)})$$
we define $\mathbb{Z}$-modules 
$$\Theta_V=\bigoplus_{v\in V(\Gamma)} \Theta_{G_v} \text{  and   }\Theta_E=\bigoplus_{e\in E(\Gamma)} \Theta_{G_e} $$
and the corresponding positive cones 
$$\Theta_V^+=\bigoplus_{v\in V(\Gamma)} \Theta_{G_v}^+ \text{  and   }\Theta_E^+=\bigoplus_{e\in E(\Gamma)} \Theta_{G_e}^+ .$$
Let us note that, since $G_e=G_{\overline{e}}$, module $\Theta_{E}$ contains each $\Theta_{G_e}$ twice. 
It is not essential, but will allow us to simplify notations. 
It will be convenient to consider these $\mathbb{Z}$-modules with the norms
$$||\cdot||_V=\frac{1}{|V(\Gamma)|}\sum_{v\in V(\Gamma)}||\cdot||_{G_v}  \text{ and }  ||\cdot||_E=\frac{1}{|E(\Gamma)|}\sum_{e\in E(\Gamma)}||\cdot||_{G_e}.$$ 

For a representation $\rho\colon\overpione\to U(W)$ (or $\rho\colon\pione\to U(W)$) we can define $\rho^\#\in\Theta^+_V$ by
$$\left(\rho^\#\right)_v=\left(\rho|_{G_v}\right)^\#$$
for $v\in V$. 

Let $$\deriv|_{\Theta_{G_v}}=\sum_{e:t(e)=v}i^\#_e-\sum_{e:o(e)=v}i^\#_{\overline{e}}.$$

\begin{proposition}
\label{basic}
For any representation $\rho \colon \pione \to U(W)$ we have $\rho^\# \in  \Theta_V^+ \cap\ker \deriv$.  
For any $\lambda \in \Theta_V^+ \cap \ker \deriv$ there exists some representation $\rho  \colon \pione \to U(W)$ with $\rho^\#=\lambda$.
\end{proposition}

\begin{proof} 
For any representation $\rho \colon \pione \to U(W)$ and any $e\in E(\Gamma)$ representations $i_e^*(\rho|_{G_{t(e)}})$ and 
$i_{\overline{e}}^*(\rho|_{G_{o(e)}})$ of $G_e$ are conjugated by $\rho(s_e)$, hence 
$$\left(\deriv (\rho^\#)\right)_e=i_e^\#(\rho^\#_{t(e)})-i^\#_{\overline{e}}(\rho^\#_{o(e)}) =
(i_e^*(\rho|_{G_{t(e)}}))^\# - (i_{\overline{e}}^*(\rho|_{G_{o(e)}}))^\#=0$$ and $\rho^\# \in \ker \deriv \cap \Theta^+_V$.

Proof of the second part has the same flavor as the proof of Remark~\ref{indep}. 
Let us note that for any $\lambda \in \Theta_V^+ \cap \ker \deriv$,  $||\lambda||_V = ||\lambda_v||_{G_v}$ for any $v\in V$, since $\Gamma$ is connected. 
Let us fix a Hilbert space $W$ with $\dim(W)=||\lambda||$. 
We will construct a representation $\rho\colon \overpione \to U(W)$ and show that relations $R_\mathcal{G}$ hold. 
At first we will construct representations $\rho_v$ of $G_v$ inductively using the spanning tree $T$.\\
\textit{Basis.} Let us fix some $v\in V(\Gamma)$ and some representation $\rho_v$ of $G_v$ with $\rho_v^\#=\lambda_v$. \\
\textit{Induction step.} Assume that for some $e\in E(T)$ we already have $\rho_{t(e)}$ with $\rho_{t(e)}^\#=\lambda_{t(e)}$. 
Let us fix a representation $\rho'\colon G_{o(e)}\to U(W)$ with $(\rho')^\#=\lambda_{o(e)}$. 
Since 
$$(i_e^*(\rho_{t(e)}))^\#=i_e^\#(\lambda_{t(e)})=i_{\overline{e}}^\#(\lambda_{o(e)})= (i_{\overline{e}}^*(\rho'))^\#,$$
these representations are isomorphic and there is a unitary intertwining operator $S$ such that 
$i_e^*(\rho_{t(e)})(g_e) \cdot S = S \cdot i_{\overline{e}}^*(\rho')(g_e)$ for all $g_e\in G_e$. 
Now we can define $\rho_{o(e)}\coloneqq ad(S^{-1})(\rho')$, where $(ad(L)(\rho))(g)=L^{-1}\cdot \rho(g) \cdot L$. 
Let us note that $i^*_{e}(\rho_{t(e)})=i^*_{\overline{e}}(\rho_{o(e)})$. \\
Now we will construct our representation on elements $s_e$, $e\in \vec{E}(\Gamma)$. 
For $e \in \vec{E}(T)$ we define $\rho_e = I$. 
For $e \in \vec{E}(\Gamma)\setminus \vec{E}(T)$ similar to the induction step we acquire unitary operators $\rho_e$ such that 
$i_e^*(\rho_{t(e)})(g_e)  \cdot \rho_e=\rho_e\cdot  i_{\overline{e}}^*(\rho_{o(e)})(g_e)$ for all $g_e\in G_e$. 
Collections of $\rho_v$ and $\rho_e$ define the representation  $\rho\colon \overpione \to U(W)$ with $\rho(s_e)=\rho_e$. The relations $R_\mathcal{G}$ hold by the construction. This finishes the proof. 
\end{proof}

%LEMMA ABOUT CONES
We will write $A\prec_{X} B$ if $A<C\cdot B$ for some constant $C$, which depends only on $X$. \\
Let us state a general fact about cones (see Lemma 5.3 in \cite{lazarovich2021virtually}) only for $\Theta^+$.
\begin{lemma} [Lazarovich, Levitt]
\label{analogue2}
For any $\lambda \in \Theta^+_V$ there exists $\lambda'' \in \Theta_V^+ \cap \ker \deriv$ satisfying 
${||\lambda-\lambda''||_V\prec_{\mathcal{G}} ||\deriv\lambda||_E}$ and $||\lambda''||_V\le||\lambda||_V$.
\end{lemma}

%FROM REPRESENTATIONS TO CONES AND BACK TO REPRESENTATIONS

\subsection{From representations to cones and back to representations}

\begin{proposition}
\label{analogue1}
Let $\rho: \overpione\to U(W)$ be a unitary representation. If $\rho$ is a $\delta$-almost $\pione$-representation for some $\delta<1$, then
$$ || \deriv(\rho^\#)||_E \prec\delta^p \cdot|| \rho^\#||_V.$$
\end{proposition}

\begin{proof}
Let us fix an oriented edge $e\in \vec{E}(\Gamma)$. 
By the definition of a $\delta$-almost $\pione$-representation and by the unitary invariance of the normalized $p$-Schatten norm we have  
$$||\rho(s_e^{-1}i_e(g)s_e)-\rho(i_{\overline{e}}(g)) ||'_p<\delta$$ 
for any  group element $g\in G_e$ . For convenience we will write $\rho_v$ instead of $\rho|_{G_v}$.
We can apply Theorem~\ref{main} to the representations ${ad(\rho(s_e))(i_e^*(\rho_{t(e)}))}$ and $i_{\overline{e}}^*(\rho_{o(e)})$. 
It gives us isomorphic subrepresentations of these representations of dimension at least $\left(1-(2\delta)^p\right)\cdot\dim W$. 
Since $ad(\rho(s_e)) (i_e^*(\rho_{t(e)}))$ and $i_{e}^*(\rho_{t(e)})$ are conjugated,
$\left(ad(\rho(s_e))(i_e^*(\rho_{t(e)}))\right)^\#=(i_e^*(\rho_{t(e)}))^\#$
and 
$$||\deriv(\rho^\#)_e||_{G_e}=||i_e^\#(\rho_{t(e)}^\#)-i_{\overline{e}}^\#(\rho_{o(e)}^\#)||_{G_e}=
||(i_e^*(\rho_{t(e)}))^\#-( i_{\overline{e}}^*(\rho_{o(e)}))^\#||_{G_e}= $$
$$= ||\left(ad(\rho(s_e))(i_e^*(\rho_{t(e)}))\right)^\#-( i_{\overline{e}}^*(\rho_{o(e)}))^\#||_{G_e} \le (2\delta)^p\cdot \dim W \prec \delta^p\cdot ||\rho^\#||_V.$$ 
This finishes the proof. 
\end{proof}

\begin{lemma}
\label{analoguelemma}
Let $i\colon H\to G$ be a homomorphism of finite groups. 
Let $\tau: H\to U(W)$ and $\rho: G \to U(W)$ be a pair of representations. 
Denote $\lambda=\rho^\#$. 
If $\lambda'\in\Theta_G^+$ and $\delta>0$ are such that 
\begin{enumerate}
\item $d_H(i^*(\rho);\tau)\le \delta$,
\item $i^\#(\lambda')=\tau^\#$,
\item $||\lambda-\lambda'||_G\le \delta^p \cdot ||\lambda||_G$,
\end{enumerate}
then there exists a unitary representation $\rho'\colon G\to U(W)$ satisfying
\begin{enumerate}
\item $i^*(\rho')=\tau$,
\item $(\rho')^\#= \lambda'$,
\item $d_G(\rho', \rho) \prec\delta$.
\end{enumerate}
\end{lemma}

\begin{proof}
We will denote by $\lambda_1\in \Theta_G^+$ the common part of $\lambda$ and $\lambda'$, that is the vector $\lambda_1$ with 
$(\lambda_1)_\pi=\min(\lambda_\pi,\lambda_\pi')$ for $\pi\in Irr(G)$.  Then
$$||\lambda-\lambda'||_G=2||\lambda-\lambda_1||_G=2||\lambda'-\lambda_1||_G.$$ 
Let us take a $\rho(G)$-invariant subspace $V_1$ such that $(\rho|_{V_1})^\#=\lambda_1$. 
Let us define a new representation $\sigma\colon G\to U(V_1^\perp)$ such that $\sigma^\#=\lambda'-\lambda_1\in \Theta^+_G$.
Now we can construct $\rho_1\colon G\to U(W)$ as a sum $\rho_1\coloneqq\rho|_{V_1}\oplus\sigma$.  
Let us note that $\rho_1^\#=\lambda'$ and, since $2\dim(V_1^\perp)\le\delta^p\cdot\dim(W)$, by Lemma~\ref{converse}
$$d_G(\rho,\rho_1)\le 2\delta.$$
Since $\rho_1|_H$ and $\tau$ are two isomorphic representations of $H$ and 
$$d_H(i^*(\rho_1),\tau)\le d_H(i^*(\rho_1), i^*(\rho))+d_H(i^*(\rho),\tau)\le d_G(\rho_1, \rho)+\delta \le3\delta,$$ 
 we can apply Corollary~\ref{maincor}. 
We get an operator $T\in U(W)$ such that $||T-I||_p'\prec \delta$ and ${ad(T)(\rho_1|_H)=\tau}$.  
Now we can define $\rho'$ by $\rho':=ad(T)(\rho_1)$. We have $i^*(\rho')=\tau$, $(\rho')^\#=\lambda'$ and
$$d_G(\rho_1,\rho')\le 2||T-I||_p'\prec\delta,$$
hence
 $$d_G(\rho',\rho)\le d_G(\rho',\rho_1)+d_G(\rho_1,\rho)\prec \delta.$$ 
This finishes the proof.
\end{proof}

To prove the next proposition we will modify the proof of Proposition~\ref{basic}.

\begin{proposition}
\label{analogue3}
Let $\rho\colon\overline{\pi}_1(\mathcal{G},T) \to U(W)$ be a $\delta$-almost $\pi_1(\mathcal{G},T)$-representation with $\lambda=\rho^\#$. Let $\lambda'\in \Theta_V^+$ be any vector with $||\lambda'||_V=||\lambda||_V$. If 
\begin{enumerate}
\item $\lambda' \in \ker \deriv$,
\item $||\lambda-\lambda'||_V\le \delta^p\cdot ||\lambda||_V$,
\end{enumerate}
then there is a representation $\rho'\colon \pi_1(\mathcal{G}, T)\to U(W)$ satisfying 
\begin{enumerate}
\item $(\rho')^\#=\lambda'$,
\item $d_{S_{\mathcal{G}}}(\rho, \rho')\prec_{\mathcal{G}}\delta$.
\end{enumerate}
\end{proposition}

\begin{proof}
We will construct a representation $\rho'$ of $\overpione$ such that $\rho'(R_{\mathcal{G}})=I$, 
so $\rho'$ can be considered as a representation of $\pione$. 
Let us note that for each $v\in V(\Gamma)$ we have ${||\lambda_v-\lambda'_v||_{G_v}\le\delta^p\cdot|V(\Gamma)|\cdot ||\lambda_v||_{G_v}}$ 
and for each $e\in E(T)$ we have $i_e^\#(\lambda'_{t(e)})=i_{\overline{e}}^\#(\lambda'_{o(e)})$.

We will start by constructing representations $\rho'_v$ of $G_v$ with $d_{G_{v}}(\rho_v,\rho'_v)\prec\delta$ and $(\rho'_v)^\#=\lambda'_{v}$ inductively using the spanning tree $T$.\\
\textit{Basis.}
Let us fix a vertex $v \in V(\Gamma)$. 
We can apply Lemma~\ref{analoguelemma} to the trivial inclusion of a one-element group into $G_{v}$, to the trivial representation of a one-element group, 
to $\rho_{v}$ and to $\lambda'_{v}\in\Theta^+_{G_{v}}$. We get a representation $\rho'_v\colon G_v\to U(W)$ such that 
$d_{G_{v}}(\rho_v,\rho'_v)\prec\delta$ and $(\rho'_v)^\#=\lambda'_{v}$. \\
\textit{Induction step.} Assume that $\rho'_{t(e)}\colon G_{t(e)}\to U(W)$ is already defined, but $\rho'_{o(e)}$ is not yet defined for some $e\in E(T)$. 
By the induction assumption we know that ${d_{G_{t(e)}}(\rho_{t(e)},\rho'_{t(e)})\prec \delta}$ and
${(\rho'_{t(e)})^\#=\lambda'_{t(e)}}$.
To construct $\rho'_{o(e)}$ we will apply Lemma~\ref{analoguelemma} to the inclusion $i_{\overline{e}}:G_e\to G_{o(e)}$, 
to the representation $ i_{e}^*(\rho'_{t(e)})$ of $G_e$, to $\rho_{o(e)}$ and to $\lambda'_{o(e)}$.  
As we mentioned already in the beginning of the proof, the second and the third conditions of Lemma \ref{analoguelemma} are satisfied. 
Moreover,
$$d_{G_e}(i_e^*(\rho'_{t(e)}), i_{\overline{e}}^*(\rho_{o(e)}))\le 
d_{G_e}(i_e^*(\rho'_{t(e)}), i_e^*(\rho_{t(e)})) +
d_{G_e}(i_e^*(\rho_{t(e)}),  i_{\overline{e}}^*(\rho_{o(e)})) \prec \delta$$
and hence the first condition is also satisfied.
The representation given by Lemma \ref{analoguelemma} satisfies\\
 ${d_{G_{o(e)}}(\rho_{o(e)},\rho'_{o(e)})\prec \delta}$ and $(\rho'_{o(e)})^\#=\lambda'_{o(e)}$, thus we proved the induction step.
Also, by the construction we have $i_e^*(\rho'_{t(e)})=i_{\overline{e}}^*(\rho'_{o(e)})$. 

To finish the proof we need to define operators $\rho'_e$ for $e \in \vec{E}(\Gamma)$, i.e. values of our new representation on $s_e$. For $e\in\vec{E}(T)$ we can define $\rho'_e\coloneqq I$. Then 
$$||\rho(s_e)-\rho'_e||'_p\le \delta.$$

For $e\in\vec{E}(\Gamma)\setminus \vec{E}(T)$ we already know that 
$$d_{G_e}\left(i_e^*(\rho_{t(e)}), i_e^*(\rho'_{t(e)})\right)\prec \delta,$$ 
$$d_{G_e}\left(i_{\overline{e}}^*(\rho_{o(e)}),i_{\overline{e}}^*(\rho'_{o(e)})\right) \prec \delta,$$ 
$$d_{G_e}\left(ad(\rho(s_e))( i_e^*(\rho_{t(e)})),i_{\overline{e}}^*(\rho_{o(e)})\right)\prec \delta,$$
hence
$$d_{G_e}\left(ad(\rho(s_e))( i_e^*(\rho'_{t(e)})),i_{\overline{e}}^*(\rho'_{o(e)})\right)\prec \delta.$$
We also know that 
$$(ad(\rho(s_e))( i_e^*(\rho'_{t(e)})))^\#=( i_e^*(\rho'_{t(e)}))^\#=(i_{\overline{e}}^*(\rho'_{o(e)}))^\#,$$
since $(\deriv \lambda')_e=( i_e^*(\rho'_{t(e)}))^\#-(i_{\overline{e}}^*(\rho'_{o(e)}))^\#=0$.
So we can apply Corollary \ref{maincor} and get $T$ such that 
$$ad(T)(ad(\rho(s_e))(i_e^*(\rho_{t(e)}'))=ad(T\rho(s_e))(i_e^*(\rho'_{t(e)}))=i_{\overline{e}}^*(\rho'_{o(e)})$$
and $||T-I||_p'\prec \delta$. 
Now we can define $\rho'_e:=T\rho(s_e)$. We have 
$$||\rho(s_e)-\rho'_e||_p' = ||T-I||_p' \prec \delta.$$ 
Collections of $\rho_v'$ and $\rho_e'$ define the representation  $\rho'\colon \overpione \to U(W)$ such that $d_{S_{\mathcal{G}}}(\rho,\rho')\prec_{\mathcal{G}} \delta$. Moreover, by the construction ${\rho'(R_{\mathcal{G}})=I}$, hence $\rho'$ can be considered as a representation of $\pione$. 
\end{proof}

Now we are ready to prove Theorem ~\ref{ourmain}. 

\begin{proof}[Proof of Theorem ~\ref{ourmain}]
Given a $\delta$-almost $\pione$-representation $\rho\colon\overpione\to U(W)$ we can consider $\lambda=\rho^\#$. 
By Proposition~\ref{analogue1} 
$$||\deriv(\lambda)||_E \prec \delta^p\cdot ||\lambda||_V.$$
By Lemma~\ref{analogue2} there exists $\lambda'' \in \Theta_V^+ \cap \ker \deriv$ such that 
$$||\lambda''-\lambda||_V\prec \delta^p \cdot||\lambda||_V  \\  \text{ and } \\ ||\lambda''||_V\le||\lambda||_V.$$
Let us note that $||\lambda||_V=\dim W$ and $||\lambda''||_V=||\lambda''_v||_{G_v}\in\mathbb{N}$ for all $v\in V$, since ${\lambda''\in\ker\deriv}$. We will construct $\lambda'\in\Theta^+_V\cap\deriv$ by 
$$\lambda'=\lambda''+(||\lambda||_V-||\lambda''||_V)\pi^\#,$$
where $\pi$ is a trivial one-dimensional representations of $\overpione$. 
We have $\lambda'\in\ker\deriv$, since $\pi^\#\in\ker\deriv$. Moreover, 
$$||\lambda'-\lambda||_V\le||\lambda''-\lambda||_V+(||\lambda||_V-||\lambda''||_V)\le 2||\lambda''-\lambda||_V\prec \delta^p\cdot||\lambda||_V.$$
So we can apply Proposition~\ref{analogue3} to $\lambda'$ and get the desired representation $$\rho'\colon \pi_1(\mathcal{G},T)\to U(W),$$
which satisfies $(\rho')^\#=\lambda'$ and $d_{S_{\mathcal{G}}}(\rho, \rho')\prec_{\mathcal{G}} \delta$. 
\end{proof}

\begin{cor}
Finitely generated virtually free groups are stable with respect to a normalized $p$-Schatten norm for $1 \leq p<\infty$. 
\end{cor}

Recall that a character of a group $G$ is a positive definite function on $G$ which is constant on conjugacy classes and takes value 1 at the unit.  A character $\tau$ is called embeddable if it factorizes through a homomorphism to a tracial ultraproduct of matrices, that is if there is a non-trivial ultrafilter $\alpha$ on $\mathbb{N}$ and a homomorphism $f\colon G \to U\left( \prod \limits_{n\in \mathbb{N}}^{\alpha} (M_n(\mathbb{C}), tr_n)  \right)$ such that $\tau_{\alpha} \circ f=\tau$, where $\tau_{\alpha}$ is a canonical trace on $\prod \limits_{n\in \mathbb{N}}^{\alpha} (M_n(\mathbb{C}), tr_n) $.
Using the necessary condition for the Hilbert-Schmidt stability (see Theorem 3 in \cite{hadwin2018stability}) we get the following corollary.

\begin{cor}
Each embeddable character of a finitely generated virtually free group is a pointwise limit of traces of finite-dimensional representations.
\end{cor}

\section*{Acknowledgements}
We thank Tatiana Shulman, Vadim Alekseev and Andreas Thom for helpful comments and remarks. The first author was supported by the Israel Science Foundation grants \#575/16 and \#957/20.

\Addresses

\end{document}